\documentclass[reqno,twoside,final]{amsart}
\usepackage{amsmath,amstext,amsthm,amsfonts,amssymb,amscd}

\newcommand{\Q}{\mathbb{Q}}

\newcommand{\Image}{{\operatorname{Im}\,}}
\newcommand{\Char}{{\operatorname{Char}\,}}
\newcommand{\D}{{\operatorname{Dist}\,}}
\newcommand{\Cent}{{\operatorname{Cent}}}
\newcommand{\tr}{{\operatorname{tr}}}
\newcommand{\discr}{{\operatorname{Discr}}}
\newcommand{\ssl}{{\operatorname{sl}}}

\newcommand{\GL}{{\operatorname{GL}}}

\newcommand{\N}{\mathbb{N}}
\newcommand{\R}{\mathbb{R}}

\theoremstyle{definition}
\newtheorem{definition}{Definition}

\newtheorem{rem}{Remark}
\newtheorem*{rem*}{Remark}
\newtheorem*{acknow*}{Acknowledgements}
\newtheorem*{examples*}{Examples}

\theoremstyle{plain}

\newtheorem{lemma}{Lemma}
\newtheorem{prop}{Proposition}
\newtheorem{theorem}{Theorem}
\newtheorem*{theorem*}{Theorem}
\newenvironment{proof-sketch}{\noindent{\bf Sketch of Proof}\hspace*{1em}}{\qed\bigskip}
\newenvironment{proof-idea}{\noindent{\bf Proof Idea}\hspace*{1em}}{\qed\bigskip}
\newenvironment{proof-of-lemma}[1]{\noindent{\bf Proof of Lemma #1}\hspace*{1em}}{\qed\bigskip}
\newenvironment{proof-of-prop}[1]{\noindent{\bf Proof of Proposition #1}\hspace*{1em}}{\qed\bigskip}
\newenvironment{proof-of-thm}[1]{\noindent{\bf Proof of Theorem #1}\hspace*{1em}}{\qed\bigskip}
\newenvironment{proof-attempt}{\noindent{\bf Proof Attempt}\hspace*{1em}}{\qed\bigskip}

\begin{document}

\title[Images of non-commutative polynomials ]{The images of non-commutative polynomials evaluated on $2\times 2$ matrices over an arbitrary field.}
\author{Sergey Malev}

\address{Department of mathematics, Bar-Ilan University,
Ramat Gan, Israel} \email {malevs@math.biu.ac.il}
\thanks{The author was
supported by an Israeli Ministry of Immigrant Absorbtion
scholarship.}
\thanks{This research was supported by the Israel Science Foundation
(grant no. 1207/12).}

\begin{abstract}
Let $p$ be a multilinear polynomial in several non-commuting
variables with coefficients in  an arbitrary field $K$. Kaplansky
conjectured that for any $n$, the image of $p$ evaluated on the
set $M_n(K)$ of $n$ by $n$ matrices is either zero, or the set of
scalar matrices, or the set $sl_n(K)$ of matrices of trace $0$, or
all of $M_n(K)$. This conjecture was proved for $n=2$ when $K$ is
closed under quadratic extensions. In this paper the conjecture is
verified for $K=\mathbb{R}$ and $n=2$, also for semi-homogeneous
polynomials  $p$, with a partial solution for an arbitrary field
$K$.
\end{abstract}

\maketitle

\section{Introduction}
This paper is the continuation of \cite{BMR1}, in which
Kanel-Belov, Rowen and the author considered the question, reputedly
raised by Kaplansky, of the possible image set $\Image p$ of a
polynomial $p$ on matrices. (L'vov later reformulated this for
multilinear polynomials, asking whether $\Image p$ is a vector subspace.)

For an arbitrary polynomial, the
question was settled for the case when $K$ is a finite field
 by Chuang \cite{Ch}, who proved that a subset $S \subseteq
M_n(K)$ containing $0$ is the image of a polynomial with constant
term zero, if and only if $S$ is invariant under conjugation.
Later Chuang's result was generalized by Kulyamin \cite{Ku1},
\cite{Ku2} for graded algebras.

For homogeneous polynomials, the question was settled for the case when the field $K$ is algebraically closed  by
\v{S}penko \cite{S}, who proved
that the union of the zero matrix and a standard open set closed under conjugation by $\GL_n(K)$ and nonzero scalar multiplication
is the image of a homogeneous polynomial.

In \cite{BMR1}  the field $K$ was required to be quadratically
closed.  Even for the field $\mathbb{R}$ of real numbers
Kaplansky's question remained   open, leading people to ask  what
happens if the field is not quadratically closed? This paper
provides a positive answer.

 The main result in this note is for
$n=2$,   settling the major part of Kaplansky's Conjecture in this case,
proving the following result (see \S\ref{def1} for terminology):

\begin{theorem}
\label{main} If $p$ is a multilinear polynomial evaluated on the
matrix ring $M_2(K)$ (where $K$ is an arbitrary field), then
$\Image p$ is either $\{0\}$, or  $K$ (the set of scalar
matrices), or $\ssl_2\subseteq\Image p$. If $K=\mathbb{R}$ then
$\Image p$ is either $\{0\}$, or  $K$, or $\ssl_2$ or $M_2$.
\end{theorem}

Also  a classification of the possible images of homogeneous
polynomials evaluated on $2\times 2$ matrices is provided:

\begin{theorem}\label{homogen}
  Let $p(x_1,\dots,x_m)$ be a semi-homogeneous polynomial evaluated on $2~\times~2$ matrices with real entries.
 Then $\Image p$ is either $\{0\}$, or the set $\R_{\geq 0}$, i.e., the matrices $\lambda I$ for $\lambda\geq 0$,
 or the set $\R$ of scalar matrices,
 or the set $\R_{\leq 0}$, i.e., the matrices $\lambda I$ for $\lambda\leq 0$,
 or the set $\ssl_{2,\geq0}(\R)$ of trace zero matrices with non-negative discriminant,
 or the set $\ssl_{2,\leq 0}(\R)$ of trace zero matrices with non-positive discriminant,
 or the set $\ssl_2(\R)$,
 or is Zariski dense in $M_2(\R)$.
\end{theorem}
\begin{rem}
 Note that in both Theorems \ref{main} and \ref{homogen} we can consider any real closed field instead of $\R$.
\end{rem}

\section{Definitions and basic preliminaries}\label{def1}
\begin{definition}
By $K\langle x_1,\dots,x_m\rangle$ we
denote the free $K$-algebra generated by noncommuting variables
$x_1,\dots,x_m$, and refer to the elements of $K\langle
x_1,\dots,x_m\rangle$ as {\it polynomials}. Consider any algebra
$R$ over a field $K$. A polynomial $p\in K\langle
x_1,\dots,x_m\rangle$ is called a {\it polynomial identity} (PI)
of the algebra $R$ if $p(a_1,\dots,a_m)=0$ for all
$a_1,\dots,a_m\in R$;  $p\in K\langle x_1,\dots,x_m\rangle$ is a
{\it central polynomial} of $R$, if for any $a_1,\dots,a_m\in R$
one has $\mbox{$p(a_1,\dots,a_m)\in \Cent(R)$}$ but $p$ is not a
PI of $R$.
A polynomial $p\in K\langle x_1,\dots,x_m\rangle$ is called {\it
multilinear} of degree $m$ if it is linear with respect to each variable. Thus, a polynomial is multilinear if
it is a polynomial of the form
$$p(x_1,\dots,x_m)=\sum_{\sigma\in S_m}c_\sigma
x_{\sigma(1)}\cdots x_{\sigma(m)},$$ where $S_m$ is the
symmetric group in $m$ letters, and $c_\sigma\in K$.
\end{definition}

We recall the following well-known lemmas (for arbitrary $n$) whose proofs can be found in
\cite{BMR1}:

\begin{lemma}[{\cite[Lemma 4]{BMR1}}]\label{graph}Let $p$ be a multilinear polynomial.
If $a_i$ are matrix units, then   $p(a_1,\dots,a_m)$ is either~$0$, or~$c\cdot e_{ij}$ for some $i\neq j$, or a diagonal matrix.
\end{lemma}
\begin{lemma}[{\cite[Lemma 5]{BMR1}}]\label{linear}Let $p$ be a multilinear polynomial.
The linear span of   $\Image p$ is either $\{0\}$, $K$, $\ssl_n$,
or $M_n(K)$. If $\Image p$ is not $\{0\}$ or $K$, then for any
$i\neq j$ the matrix unit $e_{ij}$ belongs to $\Image p$.
\end{lemma}

 We need a slight
modification of Amitsur's theorem, which also is well known:
\begin{prop}\label{Am1} The algebra of generic matrices
is a domain $D$ which can be embedded in the division algebra UD
of central fractions of Amitsur's algebra of generic matrices.
Likewise, UD contains all characteristic coefficients of $D$.
\end{prop}
\begin{proof} Any trace function can be expressed as the ratio of two
central polynomials, in view of \cite[Theorem 1.4.12]{Row}; also
see \cite[Theorem~J, p.~27]{BR} which says for any characteristic
coefficient $\alpha_k $ of the characteristic polynomial
$$\lambda^t + \sum_{k=1}^t (-1)^k \alpha _k \lambda ^{t-k}$$ that
\begin{equation}\label{trace2pol0}
\alpha_k f(a_1, \dots, a_t, r_1, \dots, r_m) = \sum _{k=1}^t
f(T^{k_1}a_1, \dots, T^{k_t} a_t,  r_1, \dots, r_m) ,
\end{equation}
summed over all vectors $(k_1, \dots, k_t)$ where each $k_i \in \{
0, 1 \}$ and $\sum k_i = t,$ where $f$ is any $t$-alternating
polynomial (and $t = n^2$). In particular,
\begin{equation}\label{trace2pol}
\tr(T)f(a_1, \dots, a_t, r_1, \dots, r_m) = \sum _{k=1}^t f(a_1,
\dots, a_{k-1}, Ta_k, a_{k+1} , \dots, a_t,  r_1, \dots, r_m) ,
\end{equation}
so any trace of a polynomial belongs to UD.
\end{proof}

We also need the First Fundamental Theorem of Invariant Theory (see \cite [Theorem $1.3$]{P})
\begin{prop}\label{procesi}
Any polynomial invariant of  $n\times n$ matrices $A_1,\dots,A_m$
is a polynomial in the invariants $\tr(A_{i_1}A_{i_2}\cdots
A_{i_k})$,  taken
  over all possible (noncommutative) products of the $A_i$.
\end{prop}

We also require one basic fact from the linear algebra:
\begin{lemma}\label{dim2}
Let $V_i$ (for $1\leq i\leq m$) and $V$ be linear spaces over arbitrary field $K$.
Let $f(T_1,\dots,T_m): \prod\limits_{i=1}^m V_i\rightarrow V$ be a multilinear mapping (i.e. linear with respect to each $T_i$).
Assume there exist two points in $\Image f$ which are not proportional.
Then $\Image f$ contains a $2$-dimensional plane. In particular, if $V$ is $2$-dimensional, then $\Image f=V$.
\end{lemma}
\begin{proof}
 Let us denote for $\mu=(T_1\dots,T_m)$ and $\nu=(T_1',\dots,T_m')\in \prod\limits_{i=1}^m V_i$
 $$\D(\mu,\nu)=\#\{i: T_i\neq T_i'\}.$$
 Consider $k=\min\{d:$ there exists $\mu,\nu\in\prod\limits_{i=1}^m V_i$ such that $f(\mu)$ is not proportional to
 $f(\nu)$ and $\D(\mu,\nu)=d\}.$
 We know $k\leq m$ by assumptions of lemma. Also $k\geq 1$ since any element of $V$ is proportional to itself.
 Assume $k=1$. In this case there exist $i$ and $T_1,\dots,T_m,T_i'$ such that $f(T_1,\dots,T_m)$ is not proportional to $f(T_1,\dots,T_{i-1},T_i',T_{i+1},\dots,T_m).$
 Therefore $$\langle f(T_1,\dots,T_m),f(T_1,\dots,T_{i-1},T_i',T_{i+1},\dots,T_m)\rangle\subseteq\Image p$$ is $2$-dimensional.
 Hence we can assume $k\geq 2$.
 We can enumerate variables and consider $\mu=(T_1,\dots,T_m)$ and $\nu=(T_1',\dots,T_k',T_{k+1},\dots,T_m)$, $v_1=f(\mu)$ is not proportional to $v_2=f(\nu)$.
 Take any $a,b\in K$.
 Consider $v_{a,b}=f(aT_1+bT_1',T_2+T_2',\dots,T_k+T_k', T_{k+1},\dots,T_m).$ Let us open the brackets.
 We have
 $$v_{a,b}=av_1+bv_2+\sum_{\emptyset\subsetneqq S\subsetneqq\{1,\dots,k\}} c_S f(\theta_S),$$
 where $c_S$ equals $a$ if $1\in S$ and $b$ otherwise,
 and $\theta_S=(\tilde T_1,\dots,\tilde T_k,T_{k+1},\dots,T_m)$ for $\tilde T_i=T_i$ if $i\in S$ or $T_i'$ otherwise.
 Note that any $\theta_S$ in the sum satisfies $\D(\theta_S,\mu)<~k$ and $\D(\theta_S,\nu)<k$ therefore
 $f(\theta_S)$ must be proportional to both $v_1$ and $v_2$ and thus $f(\theta_S)=0$.
 Therefore $v_{a,b}=av_1+bv_2$ and hence $\Image f$ contains a $2$-dimensional plane.
\end{proof}
\begin{definition}
Assume that $K$ is an arbitrary field and $F\subseteq K$ is a
subfield. The set $\{\xi_1,\dots,\xi_k\}\subseteq K$ is called
{\it  generic} (over $F$) if $f(\xi_1,\dots,\xi_k)\neq 0$ for any
commutative polynomial $f\in F[x_1,\dots,x_k]$ that takes nonzero
values.
\end{definition}
\begin{lemma}\label{many-gen}
 Assume that $K$ has infinite transcendence degree over $F$.
 Then for any $k\in\N$ there exists a set of generic elements
 $\{\xi_1,\dots,\xi_k\}\subseteq K$.
\end{lemma}
\begin{proof}
 $K$ has infinite transcendence degree over $F$. Therefore, there exists an element $\xi_1\in K\setminus \bar F$,
 where $\bar F$ is an algebraic closure of $F$.
 Now we consider $F_1=F[\xi_1]$.
 $K$ has infinite transcendence degree over $F$ and thus has infinite transcendence degree over $F_1$.
 Therefore there exists an element $\xi_2\in K\setminus \bar F_1$.
 And we consider the new base field $F_2=F_1[\xi_2]$.
 We can continue up to any natural number $k$.
\end{proof}
\begin{definition}
 We will say that a set of $n\times n$ matrices $\{x_1,\dots,x_m\}\in M_n(K)$ is {\it  generic}
over $F$ if the set of their entries $\{(x_\ell)_{i,j}|1\leq
\ell\leq m;\ 1\leq i,j\leq n\}$
 is generic.
\end{definition}
\begin{rem}
 Note that according to Lemma \ref{many-gen} if $K$ has infinite transcendence degree over $F$
 we can take as many generic elements as we need, in particular we can take as many generic matrices as we need.
\end{rem}
\begin{lemma}\label{gen-real}
 Assume $f:H\rightarrow\R$ (where $H\subseteq\R^k$ is an open set in $k$-dimensional Euclidean space)
 is a function that is continuous in a neighborhood of the point $(y_1,\dots,y_k)\in
 H$, with
 $f(y_1,\dots,y_k)<q$.
 Let    $c_i$ be   real numbers (in particular the coefficients of some
 polynomial $p$).
 Then there exists a set of  elements $\{x_1,\dots,x_k\}\subseteq\R$ generic over $F=\Q[c_1,\dots,c_N]$  such that
 $(x_1,\dots,x_k)\in H$ and $f(x_1,\dots,x_k)<q$.
\end{lemma}
\begin{proof}
We denote the $\delta$-neighborhood $N_\delta(x)$ of $x\in\R$ as
the interval
$(x-\delta,x+\delta)\subseteq\R$.
 Fix some small $\delta>0$ such that the product of $\delta$-neighborhoods of $y_k$ lays in $H$.
 For this particular $\delta$ we consider the $\delta$-neighborhood $N_\delta(y_1)$ of $y_1$:
 the interval $(y_1-\delta,y_1+\delta)$ is an uncountable set, and therefore there exists
 $x_1\in N_\delta(y_1)\setminus \bar F$. We consider $F_1=F[x_1]$ and analogically chose $x_2\in N_\delta (y_2)\setminus \bar F_1$
 and
 take $F_2=F_1[x_2]$. In such a way we can take generic elements $x_k\in N_\delta(y_k)$.
 Note that if $\delta$ is not sufficiently small $f(x_1,\dots,x_k)$ can be larger than $q$,
 but
 $$\mathop{f(x_1,\dots,x_k)\rightarrow f(y_1,\dots,y_k)}_{\delta\rightarrow 0}.$$
 Thus there exists sufficiently small $\delta$ and generic elements $x_i\in N_\delta(y_i)$ such that
 $f(x_1,\dots,x_k)<q$.
\end{proof}
\begin{rem}
 If $f(y_1,\dots,y_k)>q$, then there exists a set of generic elements $x_i\in\R$ such that $f(x_1,\dots,x_k)>q$.
\end{rem}

\begin{rem}
 Note that $f$ can be a function defined on a set of matrices. In this case we consider it as a function
 defined on the matrix entries.
\end{rem}

\section{Images of multilinear polynomials}\label{im-of-pol}
Assume that $p$ is a multilinear polynomial evaluated on $2\times
2$ matrices over any field $K$. Assume also that $p$ is neither PI
nor central. Then, according to Lemmas~\ref{graph} and
\ref{linear} there exist matrix units $a_1,\dots,a_m$ such that
$p(a_1,\dots,a_m)=e_{12}$. Let us consider the mapping $\chi$
defined on matrix units that switches the indices~$1$ and $2$,
i.e., $e_{11}\leftrightarrow e_{22}$ and $e_{12}\leftrightarrow
e_{21}$. Now let us consider the mapping $f$ defined on $m$ pairs
$T_i=(t_i,\tau_i):$
$$f(T_1,\dots,T_m)=p(t_1a_1+\tau_1\chi(a_1),t_2a_2+\tau_2\chi(a_2),\dots,t_ma_m+\tau_m\chi(a_m)).$$
Now let us open the brackets. We   showed in \cite{BMR1} (see the
proof of Lemma~8) that either all nonzero terms are diagonal, or
all nonzero terms are off-diagonal ($ce_{12}$ or $ce_{21}$). We
have the latter case, so  the image of $f$ contains only matrices
of the type $c_1e_{12}+c_2e_{21}$. Note that the matrices $e_{12}$
and $e_{21}$ both belong to the image of $f$ since
$p(a_1,\dots,a_m)=e_{12}$ and
$p(\chi(a_1),\dots,\chi(a_m))=e_{21}$.
According to Lemma \ref{dim2}
 the image
of $f$ is at least $2$-dimensional,
and lies in the $2$-dimensional plane $\langle e_{12}, e_{21}
\rangle.$ Therefore this plane is exactly the image of $f$. Now we
are ready to prove the following:
\begin{lemma}\label{genfield}
 If $p$ is a multilinear polynomial evaluated on the
 matrix ring $M_2(K)$ (for an arbitrary field  $K$),
 then $\Image p$ is either $\{0\}$, or  $K$,  or $\ssl_2\setminus K\subseteq\Image p$.
\end{lemma}
\begin{proof}
 Let $A$ be any trace zero, non-scalar matrix. Take any vector $v_1$ that is not an eigenvector of $A$.
 Consider the vector $v_2=Av_1$. Note that $Av_2=A^2v_1=-\det(A)v_1$, and therefore the matrix $A$ with respect to
 the base $\{v_1,v_2\}$ has the form $c_1e_{12}+c_2e_{21}$, for some $c_i$. Hence $A$ is similar to $c_1e_{12}+c_2e_{21} \in\Image p,$ implying $A\in\Image p$.
\end{proof}
\begin{rem}\label{chr-n2}
 Note that for $\Char(K)\neq 2$ (in particular  for $K=\mathbb{R}$),   $$(\ssl_2\setminus K) \cup\{0\}=
 \ssl_2\subseteq\Image p.$$
\end{rem}
\section{The real case}
Throughout this section we assume that $K=\mathbb{R}$. By  Lemma
\ref{genfield} we know that either $p$ is PI, or central, or
$\ssl_2\subseteq\Image p$. Assume that $\ssl_2\subsetneqq\Image
p$. We will use the following lemma:
\begin{lemma}\label{ineq}
 Let $p$ be any multilinear polynomial   satisfying $\ssl_2\subsetneqq\Image p$.
 For any $q\in\mathbb{R}$ there exist generic matrices $x_1,\dots, x_m,y_1,\dots,y_m$ such that
 for $X=p(x_1,\dots,x_m)$ and $Y=p(y_1,\dots,y_m)$ we have the following:
 $$\frac{\det X}{\tr^2 X}\leq q\leq \frac{\det Y}{\tr^2 Y},$$
 where   $\tr^2 M$ denotes the square of the trace of $M$.
\end{lemma}
\begin{proof}
 We know that $\ssl_2\subseteq\Image p$, in particular for the matrices
 $\Omega=e_{11}-e_{22}$ and $\Upsilon=e_{12}-e_{21}$
 there exist matrices $a_1,\dots,a_m,b_1,\dots,b_m$ such that
 $p(a_1,\dots,a_m)=\Omega$ and
 $p(b_1,\dots,b_m)=\Upsilon$.
 Note $\frac{\det M}{\tr^2 M}=<q$ if $M$ is close to $\Omega$ and
 $\frac{\det M}{\tr^2 M}>q$ if $M$ is close to $\Upsilon$.
 Now we consider a very small $\delta>0$ such that for any matrices $x_i\in N_\delta(a_i)$ and $y_i\in N_\delta(b_i)$
 $$\frac{\det X}{\tr^2 X}\leq q\leq \frac{\det Y}{\tr^2 Y},$$ where
 $X=p(x_1,\dots,x_m)$ and $Y=p(y_1,\dots,y_m)$.
 Here by $N_\delta(x)$ we denote a $\delta$-neighborhood of $x$, under the
 max norm $\Arrowvert A \Arrowvert=\max\limits_{i,j} \arrowvert a_{ij} \arrowvert$.
 According to Lemma \ref{gen-real} one can choose generic matrices with such property.
\end{proof}
 Now we are ready to prove that the image of $g(x_1,\dots,x_m)=\frac{\det p}{\tr^2 p}$ is everything:
 \begin{lemma}\label{anyq}
 Let $p$ be any multilinear polynomial satisfying $\ssl_2\subsetneqq\Image p$.
 Then for any $q\in\mathbb{R}$ there exists a set of matrices $a_1,\dots, a_m$ such that
 \begin{equation}\label{eq}
\frac{\det p(a_1,\dots,a_m)}{\tr^2 p(a_1,\dots,a_m)}= q.
 \end{equation}
 \end{lemma}
\begin{proof}
 Let $q$ be any real number. According to Lemma \ref{ineq}
 there exist generic matrices $x_1,\dots, x_m,y_1,\dots,y_m$ such that
 for $X=p(x_1,\dots,x_m)$ and $Y=p(y_1,\dots,y_m)$ we have the following:
 $$\frac{\det X}{\tr^2 X}\leq q\leq \frac{\det Y}{\tr^2 Y}.$$
 Consider the following matrices:
 $A_0=p(\tilde x_1,x_2, \dots,x_m)$, where
 $\tilde x_1$ is either $x_1$ or $-x_1$, such that $\tr A_0>0$.
 $A_1=p(\tilde y_1,x_2,\dots,x_m)$, where
 $\tilde y_1$ is either $y_1$ or $-y_1$ such that $\tr A_1>0$.
 Assume that $A_i$, $\tilde x_1$, $\tilde y_1,\dots,\tilde y_i$ are defined.
 Let $$A_{i+1}=p(\tilde y_1,\dots,\tilde y_i,\tilde y_{i+1},x_{i+2},\dots,x_m)$$ where
 $\tilde y_{i+1}=\pm y_{i+1}$ is such that $\tr A_{i+1}>0$.
 In such a way we defined matrices $A_i$ for $0\leq i\leq m$.
 Note that for any $2\times 2$ matrix~$M$,
 $$\frac{\det M}{\tr^2 M}=\frac{\det (-M)}{\tr^2 (-M)}$$
 Note that $A_0=\pm p(x_1,\dots,x_m)$ and $A_m=\pm p(y_1,\dots, y_m);$ hence
 $$\frac{\det A_0}{\tr^2 A_0}\leq q\leq \frac{\det A_m}{\tr^2 A_m}.$$
 Therefore there exists $i$ such that
 $$\frac{\det A_i}{\tr^2 A_i}\leq q\leq \frac{\det A_{i+1}}{\tr^2 A_{i+1}}.$$
Since
 $A_{i}=p(\tilde y_1,\dots,\tilde y_i,x_{i+1},x_{i+2},\dots,x_m)$ and
 $A_{i+1}=p(\tilde y_1,\dots,\tilde y_{i+1},x_{i+2},\dots,x_m)$,
  we can consider the matrix function
 $$M(t)=(1-t)A_i+tA_{i+1}=p(\tilde y_1,\dots,\tilde y_i,(1-t)x_{i+1}+t\tilde y_{i+1},x_{i+2},\dots,x_m),$$
 Then $\Image M\subseteq\Image p,$ $M(0)=A_i$, $M(1)=A_{i+1}$ both $M(0)$ and $M(1)$ have positive trace, and
 $M$ is an affine function. Therefore for any $t\in [0,1]$ $M(t)$ has positive trace.
 Therefore the function $\psi(t)=\frac{\det M(t)}{\tr^2 M(t)}$ is well defined on $[0,1]$ and continuous.
 Also we have $\psi(0)\leq q\leq \psi(1)$. Thus there exists $\tau\in [0,1]$ such that $\psi(\tau)=q$ and thus
 $M(\tau)\in\Image p$ satisfies equation
 \eqref{eq}.
\end{proof}
\begin{lemma}\label{discr-not-zero}
 Let $p$ be a multilinear polynomial satisfying $\ssl_2\subsetneqq\Image p$.
 Then any matrix with distinct eigenvalues (i.e. matrix of nonzero discriminant) belongs to $\Image p.$
\end{lemma}
\begin{proof}
 Let $A$ be any matrix with nonzero discriminant. Let us show that $A\in\Image p$.
 Let $q=\frac{\det A}{\tr^2 A}$.
 According to Lemma \ref{anyq}
 there exists a set of matrices $a_1,\dots, a_m$ such that
 $\frac{\det \tilde A}{\tr^2 \tilde A}= q,$ where $\tilde A=p(a_1,\dots,a_m)$.
 Take $c\in\R$ such that $\tr(c\tilde A)=\tr A$. Note $c\tilde A=p(ca_1,a_2,\dots,a_m)$ belongs to $\Image p.$
 Thus $$\frac{\det (c\tilde A)}{\tr^2 (c\tilde A)}=q=\frac{\det A}{\tr^2 A},$$ and $\tr A=\tr(c\tilde A)$.
 Hence, $\det(c\tilde A)=\det(A)$. Therefore the matrices $c\tilde A$ and $A$ are similar since
 they are not from the discriminant surface. Therefore $A\in\Image p$.
\end{proof}
\begin{lemma}\label{unip}
 Let $p$ be a multilinear polynomial satisfying $\ssl_2\subsetneqq\Image p$.
 Then any non-scalar matrix with zero discriminant belongs to $\Image p.$
\end{lemma}
\begin{proof}
 Let $A$ be any non-scalar matrix with zero discriminant. Let us show that $A\in\Image p$.
 The eigenvalues of $A$ are equal, and therefore they must be real.
 Thus $A$ is similar to the matrix
 $\tilde A = \left(
 \begin{matrix}
 \lambda & 1
 \\  0 & \lambda
 \end{matrix}
 \right)
 .$
 If $A$ is nilpotent then $\lambda=0$ and $\tilde A=e_{12}$, and it belongs to $\Image p$  by Lemmas \ref{graph} and \ref{linear}.
 If $A$ is not nilpotent then we need to prove that at least one non-nilpotent matrix of such type belongs to $\Image p,$ and
 all other are similar to it.
 We know that the matrices $e_{11}-e_{22}=p(a_1,\dots,a_m)$ and $e_{12}-e_{21}=p(b_1,\dots,b_m)$ for some $a_i$ and $b_i$.
 Note that $e_{11}-e_{22}$ has positive discriminant and $e_{12}-e_{21}$ has negative discriminant.
 Take generic matrices $x_1,x_2,\dots,x_m, y_1,\dots,y_m$ such that $x_i\in N_\delta(a_i)$ and $y_i\in N_\delta(b_i)$ where
 $\delta>0$ is so small that $p(x_1,\dots,x_m)$ has positive discriminant and $p(y_1,\dots,y_m)$ has negative discriminant.
 Consider the following matrices:
 $$A_0=p(x_1,x_2, \dots,x_m) ,\qquad  A_i=p(y_1,\dots,y_i,\ x_{i+1},\dots,x_m), 1 \le i \le m.$$

 We know that $\discr A_0>0$ and $\discr A_m<0$, and therefore there exists $i$ such that $\discr A_i>0$ and $\discr A_{i+1}<0$.
 We can consider the continuous matrix function
 $$M(t)=(1-t)A_i+tA_{i+1}=p(y_1,\dots,y_i,(1-t)x_{i+1}+ty_{i+1},x_{i+2},\dots,x_m).$$
 We know that $M(0)$ has positive discriminant and $M(1)$ has negative discriminant.
 Therefore for some $\tau$, $M(\tau)$ has discriminant zero.
 Assume there exists $t$ such that $M(t)$ is nilpotent.
 In this case either $t$ is unique or there exists  $t'\neq t$ such that $M(t')$ is also nilpotent.
 If $t$ is unique then it equals to some rational function with respect to other variables (entries of matrices $x_i$ and $y_i$).
 In this case $t$ can be considered as a function on matrices $x_i$ and $y_i$ and as soon as it is invariant, according to
 the Proposition \ref{procesi} $t$ is an element of UD
 and thus
 $M(t)$ is the element of UD. Therefore $M(t)$ cannot be nilpotent since UD is a domain according to Proposition \ref{Am1}.
 If there exists  $t'\neq t$ such that $M(t')$ is also nilpotent then for any $\tilde t\in \R$ $M(\tilde t)$ is the combination
 of two nilpotent (and thus trace vanishing) matrices $M(t)$ and $M(t')$. Hence $M(0)$ is trace vanishing and thus
 $\Image p\subseteq\ssl_2$, a contradiction.

 Recall that we proved $M(\tau)$ has discriminant zero  that for some $\tau$. Note that $M(\tau)$ cannot be nilpotent.
 Assume that the matrix $M(\tau)$ is scalar.
 Hence $(1-\tau)A_i+\tau A_{i+1}=\lambda I$ where $\lambda\in\R$ and $I$ is the identity matrix.
 Thus,
 $A_{i+1}=\frac{1-\tau}{\tau}A_i+cI$.
 Note that for any matrix $M$ and any $c\in\R$ we have $\discr(M)=\discr(M+cI)$.
 Therefore the discriminant of $A_{i+1}$ can be written as
 $$\discr(A_{i+1})=\discr\left(\frac{1-\tau}{\tau}A_i\right)=\left(\frac{1-\tau}{\tau}\right)^2\discr(A_i),$$
 a contradiction, since $\discr A_i>0$ and $\discr(A_{i+1})<0$.
 Therefore the matrix $M(\tau)$ is similar to $A$.
\end{proof}
\begin{lemma}\label{scalar}
 Let $p$ be a multilinear polynomial satisfying $\ssl_2\subsetneqq\Image p$.
 Then every scalar matrix belongs to $\Image p.$
\end{lemma}
\begin{proof}
 Note that it is enough to show that at least one scalar matrix belong to the image of $p$.
 According to Lemmas \ref{graph} and \ref{linear} there are matrix units $a_1,\dots,a_m$ such that $p(a_1,\dots,a_m)$ is
 diagonal with nonzero trace.
 Assume that it is not scalar, i.e., $p(a_1,\dots,a_m)=\lambda_1e_{11}+\lambda_2e_{22}.$
 We define again the mapping $\chi$ and $f(T_1,\dots,T_m)$ as in the beginning of $\S\ref{im-of-pol}$ and return to the
 proof of Lemma~8 in \cite{BMR1} where we proved that $\Image f$ consists only of diagonal matrices or only of matrices with zeros
 on the diagonal.
 In our case the image of $f$ consists only of diagonal matrices, which is a $2$-dimensional variety.
 We know that both $p(a_1,\dots,a_m)=\lambda_1e_{11}+\lambda_2e_{22}$ and
 $p(\chi(a_1),\dots,\chi(a_m))=\lambda_1e_{22}+\lambda_2e_{11}$ belong to the image of $f$, and
  therefore every diagonal matrix belong
 to the image of $f$, in particular every scalar matrix.
\end{proof}
Now we are ready to prove the main theorem.
\\[2\baselineskip]
\begin{proof-of-thm}{\ref{main}}
 The second part follows from Lemmas \ref{genfield}, \ref{discr-not-zero}, \ref{unip} and \ref{scalar}.
 In the first part we need to prove that if $p$ is neither PI nor central then $\ssl_2(K)\subseteq\Image p$.
 According to Lemma \ref{genfield}, $\ssl_2(K)\setminus K\subseteq\Image p$, and therefore according to Remark \ref{chr-n2} we need consider only the case $\Char(K)=2$.
 In this case we need to prove that the scalar matrices belong to the image of $p$.
 According to Lemmas \ref{graph} and \ref{linear} there are matrix units $a_1,\dots,a_m$ such that $p(a_1,\dots,a_m)$ is diagonal.
 Assume that it is not scalar. Then we  consider the mappings $\chi$ and $f$ as described in the beginning of~$\S\ref{im-of-pol}$.
 According to Lemma \ref{dim2} the image of $f$ will be the set of all diagonal matrices, and in particular the scalar matrices belong to it.
\end{proof-of-thm}
\begin{rem}
 Assume that $p$ is a multilinear polynomial evaluated on $2\times 2$ matrices over an arbitrary infinite field $K$. Then, according to
 Theorem \ref{main}, $\Image p$ is $\{0\}$, or $K$, or $\ssl_2(K)$ or $\ssl_2(K)\subsetneqq\Image p$. In the last case it is clear that
 $\Image p$ must be Zariski dense in~$M_2(K)$, because otherwise $\dim(\Image p)=3$ and $\Image p$ is reducible, a contradiction.
\end{rem}

\begin{rem}
 Note that the proof of Theorem \ref{main} does not work when $n>2$ since for this case we will need to take more than one function (two
 functions for $n=3$ and more for $n>3$).
 In our proof we used that we have only one function: we proved that it takes values close to $\pm\infty$ and after that used
 continuity.  This does not work for $n\geq 3$. However one can use this idea for the question of possible images
 of trace zero multilinear polynomials evaluated on $3\times 3$ matrices. In this case one function will be enough,
 and one can take $g=\frac{\omega_3^2}{\omega_2^3}$. (One can find the definitions of~$\omega_i$
 in the proof of Theorem 3 in \cite{BMR2}.)
 Moreover according to Lemmas \ref{graph} and \ref{linear} there are matrix units $a_i$ such that $p(a_1,\dots,a_m)$ is
 a diagonal, trace zero, nonzero real matrix, which cannot be $3$-scalar since it will have three real eigenvalues.
 Therefore $p$ cannot be $3$-central polynomial.
 However the question of possible images of $p$ remains being an open problem.
\end{rem}

\section{Images of semi-homogeneous polynomials evaluated on $2\times 2$ matrices with real entries.}
Here we provide a classification of the possible images of
semi-homogeneous polynomials evaluated on $2\times 2$ matrices
with real entries. Let us start with the definitions.

\begin{definition}
A polynomial $p$ (written as a sum of monomials) is called {\it
semi-homogeneous of weighted degree $d\neq 0$} with (integer) {\it
weights} $(w_1,\dots,w_m)$  if for each monomial $h$ of $p$,
taking $d_j$ to be the degree of $x_{j}$ in $p$, we have
$$d_1w_1+\dots+d_nw_n=d.$$ A
semi-homogeneous polynomial with weights $(1,1,\dots, 1)$ is
called $\it{homogeneous}$ of degree $d$.

A polynomial $p$ is {\it completely homogeneous} of multidegree
$(d_1,\dots,d_m)$ if each variable $x_i$ appears the same number
of times $d_i$ in all monomials.
\end{definition}

\begin{definition}
A {\it cone} of $M_n(\R)$ is a subset closed under multiplication
by nonzero constants. An  {\it invariant cone} is a cone invariant
under conjugation. An invariant cone  is  {\it irreducible} if it
does not contain any nonempty  invariant cone.
A {\it semi-cone} of $M_n(\R)$ is a subset closed under multiplication
by positive constants.
An  {\it invariant semi-cone} is a semi-cone invariant
under conjugation. An invariant semi-cone  is  {\it irreducible} if it
does not contain any nonempty  invariant semi-cone.
\end{definition}
\begin{rem}
 Note that any cone is a semi-cone.
\end{rem}

\begin{rem}\label{semcone}
Let $p$ be any semi-homogeneous polynomial of weghted degree $d\neq 0$ with weights $(w_1,\dots,w_m)$.
Thus if $A=p(x_1,\dots,x_m)$ then for any $c\in\R$ we have $p(c^{w_1}x_1,\dots,c^{w_m}x_m)=c^dA$
therefore if $d$ is odd then $\Image p$
is a cone, and if $d$ is even, $\Image p$ is a semi-cone. Hence for any $d$ $\Image p$ is a semi-cone.
\end{rem}
\textbf{Theorem \ref{homogen}.} \emph{Let $p(x_1,\dots,x_m)$ be a
semi-homogeneous polynomial.
 Then $\Image p$ is either $\{0\}$, or the set $\R_{\geq 0}$, i.e., the matrices $\lambda I$ for $\lambda\geq 0$,
 or the set $\R_{\leq 0}$, i.e., the matrices $\lambda I$ for $\lambda\leq 0$,
 or the set $\R$ of scalar matrices,
 or the set $\ssl_{2,\geq0}(\R)$ of trace zero matrices with non-negative discriminant,
 or the set $\ssl_{2,\leq 0}(\R)$ of trace zero matrices with non-positive discriminant,
 or the set $\ssl_2(\R)$,
 or Zariski dense in~$M_2(\R)$.}

\begin{proof}
 Consider the function $g(x_1,\dots,x_m)=\frac{\det p}{\tr^2 p}$.
 If this function is not constant, then $\Image p$ is Zariski dense.
 Assume that it is constant; i.e., $\frac{\det p}{\tr^2 p}=c$.
 Then the ratio $\frac{\lambda_1}{\lambda_2}=\hat c$ of eigenvalues is also a constant.
 If $\hat c\neq -1$ then we can write $\lambda_1$ explicitly as
 $$\lambda_1=\frac{\lambda_1}{\lambda_1+\lambda_2}\tr p=\frac{1}{1+\frac{\lambda_2}{\lambda_1}}\tr p
 =\frac{1}{1+\frac{1}{\hat c}}\tr p,$$
 Therefore $\lambda_1$ is an element of UD, and $\lambda_2=\tr p-\lambda_1$ also.
 According to the Hamilton-Cayley equation, $(p-\lambda_1)(p-\lambda_2)=0$ and therefore, since, by Proposition~\ref{Am1},
 UD is a domain,
 one of the terms $p-\lambda_i$ is a PI. Therefore $p$ is central or PI.
 Therefore we see that any semi-homogeneous polynomial is either PI, or central, or trace zero (if the ratio of eigenvalues is
 $-1$ then the trace is identically zero), or $\Image p$ is Zariski dense.
 If $p$ is PI then $\Image p=\{0\}$.
 If $p$ is central then, by Remark~\ref{semcone}, $\Image p$ is a semi-cone, therefore
 $\Image p$ is either $\R_{\geq 0}$, or $\R_{\leq 0}$, or $\R$.
 If $\Image p$  has trace zero, then   any trace zero matrix $A\in\ssl_2(\R)$ is similar to $-A$.
 Therefore $\Image p=-\Image p$ is symmetric.
 Together with Remark \ref{semcone} we have that $\Image p$ must be a cone.
 The determinant cannot be identically zero since otherwise the polynomial is nilpotent,   contrary to Proposition \ref{Am1}.
 Hence there exists some value with nonzero determinant.
 All the trace zero matrices of positive determinant are pairwise similar, and
 all the trace zero matrices of negative determinant are pairwise similar.
 Therefore in this case all possible images of $p$ are $\ssl_{2,\geq0}(\R)$, $\ssl_{2,\leq0}(\R)$ and $\ssl_{2}(\R)$.
\end{proof}

\begin{examples*}
 $\Image p$ can be the set of non-negative scalars. Take any central polynomial, say $p(x,y)=[x,y]^2$ and consider
 $p^2=[x,y]^4$. If one  takes $-p^2=-[x,y]^4$, then its image
 is the set $\R_{\leq 0}$.

 The question remains   open of whether  or not there exists an example of a trace zero polynomial with non-negative (or non-positive) discriminant.

 There are many polynomials with Zariski dense image which are not dense with respect to the usual Euclidean topology.
 For example the image of the polynomial $p(x)=x^2$  is the set of matrices with two positive eigenvalues,
 or two complex conjugate eigenvalues; in particular any matrix $x^2$ has non-negative determinant.
The image of  the polynomial $p(x,y)=[x,y]^4+[x^4,y^4]$ is the set
of matrices with non-negative trace.
 The question of classifying   possible semi-homogeneous Zariski dense images is not simple, and
 also remains   open.
\end{examples*}

\end{document}